%% file: root.tex
\documentclass[10pt,twocolumn,twoside]{IEEEtran}  

\IEEEoverridecommandlockouts                              

\usepackage{graphicx}
\usepackage{amsmath,amsthm,amssymb}
\usepackage{enumerate}

\usepackage{tikz}
\usetikzlibrary{automata,arrows,positioning,calc} 

\usepackage{mathtools}

\usepackage{booktabs}

\usepackage[font=footnotesize]{caption}

\newtheorem{theorem}{Theorem}

\newtheorem{corollary}{Corollary}
\newtheorem{definition}{Definition}

\newtheorem{assumption}{Assumption}
\newtheorem{property}{Property}

\newcommand{\roadset}{{[}n{]}} 
\newcommand{\vtypeset}{{[}m{]}}
\newcommand{\flowdemand}{\bar{f}}
\newcommand{\sharedroads}{\mathcal{N}}
\newcommand{\typesonroad}{\mathcal{M}}
\newcommand{\prevflow}{{f^*}}
\newcommand{\newflow}{{\tilde{f}}}
\newcommand{\largenum}{{P}}

\DeclareMathOperator*{\argmin}{arg\,min}

\title{\LARGE \bf
	Optimal Tolling for Multitype Mixed Autonomous Traffic Networks
}

\author{Daniel A.~Lazar and Ramtin~Pedarsani
	\thanks{The authors are with the Department of Electrical and Computer Engineering, 
		University of California, Santa Barbara
		{\tt\small \{dlazar, ramtin\}@ucsb.edu}}
}

\begin{document}
	
\maketitle
\thispagestyle{empty}
\pagestyle{empty}

\begin{abstract}
When selfish users share a road network and minimize their individual travel costs, the equilibrium they reach can be worse than the socially optimal routing. Tolls are often used to mitigate this effect in traditional congestion games, where all vehicle contribute identically to congestion. However, with the proliferation of autonomous vehicles and driver-assistance technology, vehicles become heterogeneous in how they contribute to road latency. This magnifies the potential inefficiencies due to selfish routing and invalidates traditional tolling methods. To address this, we consider a network of parallel roads where the latency on each road is an affine function of the quantity of flow of each vehicle type. We provide tolls (which differentiate between vehicle types) which are guaranteed to minimize social cost at equilibrium. The tolls are a function of a calculated optimal routing; to enable this tolling, we prove that some element in the set of optimal routings has a lack of \emph{cycles} in a graph representing the way vehicles types share roads. We then show that unless a planner can differentiate between vehicle types in the tolls given, the resulting equilibrium can be unboundedly worse than the optimal routing, and that marginal cost tolling fails in our setting.
\end{abstract}

\section{INTRODUCTION}
\input{./intro.tex}

\section{MODEL}\label{sct:model}
\input{./model.tex}

\section{TOLLING}\label{sct:results}
\input{./results.tex}

\section{NECESSITY OF TOLLING}\label{sct:bounds}
\input{./bounds.tex}

\section{NUMERICAL EXAMPLE}\label{sct:example}
\input{./example.tex}

\section{CONCLUSION}\label{sct:conclusion}
\input{./conclusion.tex}

\bibliographystyle{IEEEtran}
\bibliography{refs}


\end{document}

%% file: intro.tex
How autonomous vehicles will change the efficiency of traffic networks is still ambiguous. While the platooning capabilities of autonomous vehicles may increase the capacity of roads up to three-fold \cite{lioris2017platoons}, when users choose their routes selfishly, the presence of capacity-improving autonomous vehicles may worsen congestion \cite{mehr2018can, brown2019tragedy}, even beyond the selfish equilibria which emerge in the presence of only human drivers \cite{roughgarden2002bad}. Prior works have studied how to use tolling to mitigate these effects in networks with a single vehicle type \cite{beckmann1956studies} or for roads shared between human drivers and autonomous vehicles which are uniform in their autonomous capabilities \cite{mehr2019pricing,lazar2019optimal}.

However, currently there are many different vehicles on the market with different levels of autonomy, including multiple Adaptive Cruise Control (ACC) systems, which affect road congestion differently \cite{gunter2020commercially}. Moreover, even within human-driven vehicles, different types of vehicles vary in how they affect congestion. To adequately understand and control traffic networks, models must incorporate multiple vehicle types. Because of this, we consider tolling for a network shared between multiple vehicle types, each affecting road latency differently. This is a setting for which no tolling results yet exist, absent extremely restrictive assumptions \cite{dafermos1972multiclass_user}. 

We consider a network of parallel roads with an arbitrary number of vehicle types, where the latency on each road is an affine function of the flow of each vehicle type on the road. We provide a theoretical property of optimal routing in this setting and use this property to establish optimal tolls. We then show that tolls must be \emph{differentiated}, meaning the social planner must be able to levy different tolls to each vehicle type on a road. We further show the failure of a classic tolling scheme, and conclude with a numerical example of our scheme.

We summarize our contributions as follows.
\begin{enumerate}
\item We derive a theoretical property of the set of optimal routing,
\item We use this property to design tolls which guarantee that the only existing selfish equilibrium also minimizes the social cost, and
\item We show the possible failure of nondifferentiated tolls and marginal cost tolling.
\end{enumerate}

\textbf{Previous Work. }
Our work builds on the field of \emph{congestion games}, which studies routing of vehicles over transportation networks, where each road is defined as an edge of a graph, where the latency experienced by users of that edge is a (typically increasing) function of the vehicle flow on that edge. There are several relevant focuses of study in this field. The first is optimal traffic assignment, which studies how to route vehicle flow in a manner which minimizes the social cost, typically understood to be the aggregate latency experienced by all users \cite{dafermos1969traffic_general}. Another is understanding equilibria which arise when all users are self-interested and choose routes to minimizes their individual travel latency \cite{wardrop1900some,depalma1998optimization}. Another vein of research is understanding and bounding the gap between social cost when users are routed optimally with respect to the social cost compared to when users choose routes selfishly \cite{roughgarden2002bad}. Many works study tolling and seek to find \emph{optimal tolls}, which, when applied, make it so that the only equilibria that exist will minimize the social cost \cite{beckmann1956studies}. Other works seek to improve the efficiency of equilibria by persuading drivers with route recommendations \cite{zhu2019routing,wu2019information}.

Some previous works extend these topics to setting where each road has multiple vehicle types which affect congestion differently, dealing with traffic assignment \cite{dafermos1972multiclass_user}, equilibria \cite{correa2008geometric}, bounding the gap between optimal and equilibria costs \cite{perakis2007price, lazar2020routing}, and tolling \cite{dafermos1973toll}. However, the cited works on traffic assignment, equilibria, and tolling have a restrictive critical assumption that is violated in the general setting, including in the model developed below for mixed autonomy. A preliminary version of this work studies our setting when there are only two vehicle types on each road \cite{lazar2019optimal}. Also in the setting of two vehicle types, \cite{mehr2018can} shows that seemingly paradoxically, converting human-driven vehicles to more efficient autonomous vehicles can worsen social cost at equilibrium; in another work the authors bound this effect \cite{mehr2019will}. In \cite{mehr2019pricing}, the authors provide optimal tolls for general networks with multiple source-destination pairs in the homogeneous case, where the difference in how each of the two vehicle types affect congestion is constant across all roads in the network. However, none of the prior works find optimal tolls for traffic networks with more than two vehicle types; in light of this we find optimal tolls for parallel networks with affine latency functions with no restrictions on the form of the affine latency functions aside from requiring it to be increasing with respect to the flow of each vehicle type.

%% file: model.tex
We consider a network of $n$ parallel roads with vehicle flow demand from $m$ vehicle types. We use $\roadset = \{ 1,2,\ldots,n \}$ and $\vtypeset = \{ 1,2,\ldots,m \}$ to denote the set of roads and vehicle types, respectively. In general, for an integer $x$, we define ${[}x{]}= \{1,2,\ldots,x\}$. We generally use $i$ to index a road and $j$ to index a vehicle type. We consider nonatomic vehicle flow and use $f^j_i \ge 0$ to denote the magnitude of vehicle flow of type $j \in \vtypeset$ on road $i \in \roadset$. We consider inelastic flow demand, where each vehicle type $j$ has flow demand $\flowdemand^j$; a feasible routing $f$ is one such that $\sum_{i \in \roadset}f^j_i = \flowdemand^j$ for all $j \in \vtypeset$ and $f^j_i\ge 0$ for all $i \in \roadset$ and $j \in \vtypeset$. 
\begin{definition}\label{def:feas}
	Let $\mathcal{F}_{\flowdemand}$ denote the set of \emph{feasible routings} for flow demand vector $\flowdemand$. Then, $\mathcal{F}_{\flowdemand} = \{ (f^j_i)_{i \in \roadset, j \in \vtypeset} : f^j_i \ge 0 \land i \in \roadset \land j \in \vtypeset \land \sum_{i\in \roadset}f^j_i = \flowdemand^j \; \forall j \in \vtypeset  \} $.
\end{definition}

We define the flow vector as
$$f = \begin{bmatrix}f^1_1, f^2_1, \ldots, f^m_1, f^1_2, f^2_2, \ldots, f^m_n \end{bmatrix}^T \; .$$
We also define the flow vector on road $i$ as $f_i =  \begin{bmatrix}f^1_i, f^2_i, \ldots, f^m_i \end{bmatrix}^T$.

Each road $i$ has a \emph{latency function} that is experienced identically by all vehicles on the road. We consider vehicle types with different autonomous technologies allowing them each to maintain different headway to the vehicle in front of it. Let $h_i^j$ denote the nominal space used by vehicle of type $j$ on road $i$, where the nominal space includes the length of the vehicle and its nominal headway. We then model road capacity to be inversely related to the average space occupied by a vehicle on a road \cite{lazar2017routing,askari2017effect}:
$$q_i(f_i) = v_id_i\sum_{j\in\vtypeset}f^j_i /(\sum_{j\in\vtypeset}f^j_i h^j_i)\; , $$
where $v_i$ and $d_i$ respectively denote the free-flow length of the road and the length of the road. Using this in conjunction with the Bureau of Public Roads latency model \cite{bureau1964manual}, we find latency function
$$\ell_i(f_i) =  t_i(1+\rho_i(\sum_{j \in \vtypeset}h^j_i f^j_i)^{\sigma_i}) \; , $$
where $t_i$ denotes the free-flow latency and $\rho_i$ and $\sigma_i$ are model parameters. Choosing $\sigma_i=1$ for all $i \in \roadset$ and letting $a^0_i = t_i$ and $a^j_i=\rho_i h^j_i$, we derive our latency function:
\begin{equation}\label{eq:optimization}
	\ell_i(f_i) = a^0_i + \sum_{j \in \vtypeset}a^j_i f^j_i \; .
\end{equation}

As mentioned above, $a^0_i$ denotes the free-flow latency of road $i$ and $a^j_i$ denotes the scaling by which the latency on road $i$ increases with the addition of vehicle type $j$. Accordingly, $a^0_i \ge 0$ and $a^j_i \ge 0$ for all roads and vehicle types.

\begin{assumption}\label{asmp:increasing_cost}
	The latency of each road is strictly increasing with the flow of each vehicle type on that road. Mathematically, $\frac{\partial \ell_i}{\partial f^j_i}(f) > 0 \; \forall i \in \roadset, \forall j \in \vtypeset$. This is equivalent to the condition $a^j_i>0$ for all $i\in \roadset$, $j \in \vtypeset$.
\end{assumption}

We wish to minimize the social cost, which we consider to be the total latency experienced by all users of the network:
\begin{equation}\label{eq:objective}
J(f) = \sum_{i \in \roadset} (\sum_{j \in \vtypeset} f^j_i) \ell_i(f_i) \; .
\end{equation}

We will later derive properties of the set of optimal routings
\begin{equation}\label{eq:optimization}
	\mathcal{F}^* = \argmin_{f \in \mathcal{F}_{\flowdemand}}J(f) \; .
\end{equation}

Note that in general the optimal routing is not unique, and multiple different routing choices can yield the same minimum cost. We do not make claims about all routings that satisfy \eqref{eq:optimization}, rather we will make claims that apply to at least one routing in the set of routings that minimize the social cost (\emph{i.e.} at least one routing in the set $\mathcal{F}^*$).

In addition to considering the socially optimal routings, we also consider how selfish users will choose their routes. To influence this user choice, we levy tolls, where tolls for different user types can differ on a given road. We model user type $j$ on road $i$ as experiencing cost
\begin{equation}
	c^j_i(f) = \ell_i(f) + \tau^j_i \; ,
\end{equation}
where $\tau^j_i$ is the toll levied on user type $j$ on road $i$. Note that only the toll can make users of different types experience different costs on the same road; we assume all users experience road latency identically. Also note that tolls are considered to be circulated back into the public coffers and are therefore not included in the social cost \eqref{eq:objective}. We consider users who are myopic and selfish; we therefore model users as following a Nash Equilibrium.
\begin{definition}\label{def:NE}
	A flow $f$ is a \emph{Nash Equilibrium} if $f^j_i > 0$ implies $c^j_i(f) \le c^j_{i'}(f)$ for all $i, i' \in \roadset$, $j \in \vtypeset$.
\end{definition}
Since we consider a networks of parallel roads, a flow is at Nash Equilibrium if no user can decrease their cost by switching roads.

We define some notation to make it easier to discuss properties of specific routings. For a specific routing $f$, we use $\sharedroads^f_j$ to denote the set of roads with positive flow of vehicle type $j$:
\begin{equation*}
	\sharedroads^f_j = \{i : f^j_i > 0 \land i \in \roadset \} \; .
\end{equation*}
Similarly, we use $\typesonroad^f_i$ to denote the set of vehicle types with positive flow on road $i$ for the routing $f$:
\begin{equation*}
	\typesonroad^f_i = \{j : f^j_i > 0 \land j \in \vtypeset \} \; .
\end{equation*}

The theoretical results established in the next section conceptualize vehicle flow on roads in the form of a graph, where for each specific routing $f$, a graph can be constructed. We construct a bipartite graph $G=(U,V,E)$ where one set of nodes is the set of roads ($U = \roadset$) and one set of nodes is the set of vehicle types ($V=\vtypeset$). The set of edges connect vehicle types to roads on which they have positive flow, \emph{i.e.}
\begin{equation}\label{eq:graph}
	E = \{ (i, j) : i \in \roadset \land  j \in \typesonroad^f_i \} \; ,
\end{equation}
or equivalently, $E = \{ (i, j) : f^j_i>0 \land i \in \roadset \land j \in \vtypeset \}$. In other words, for a routing $f$, there is an edge between the nodes denoting road $i$ and vehicle type $j$ if there is positive flow of type $j$ on road $i$. We illustrate this in Figure~\ref{fig:cyclic_acyclic}.

\begin{figure}
	\centering
	\includegraphics[width=1\linewidth]{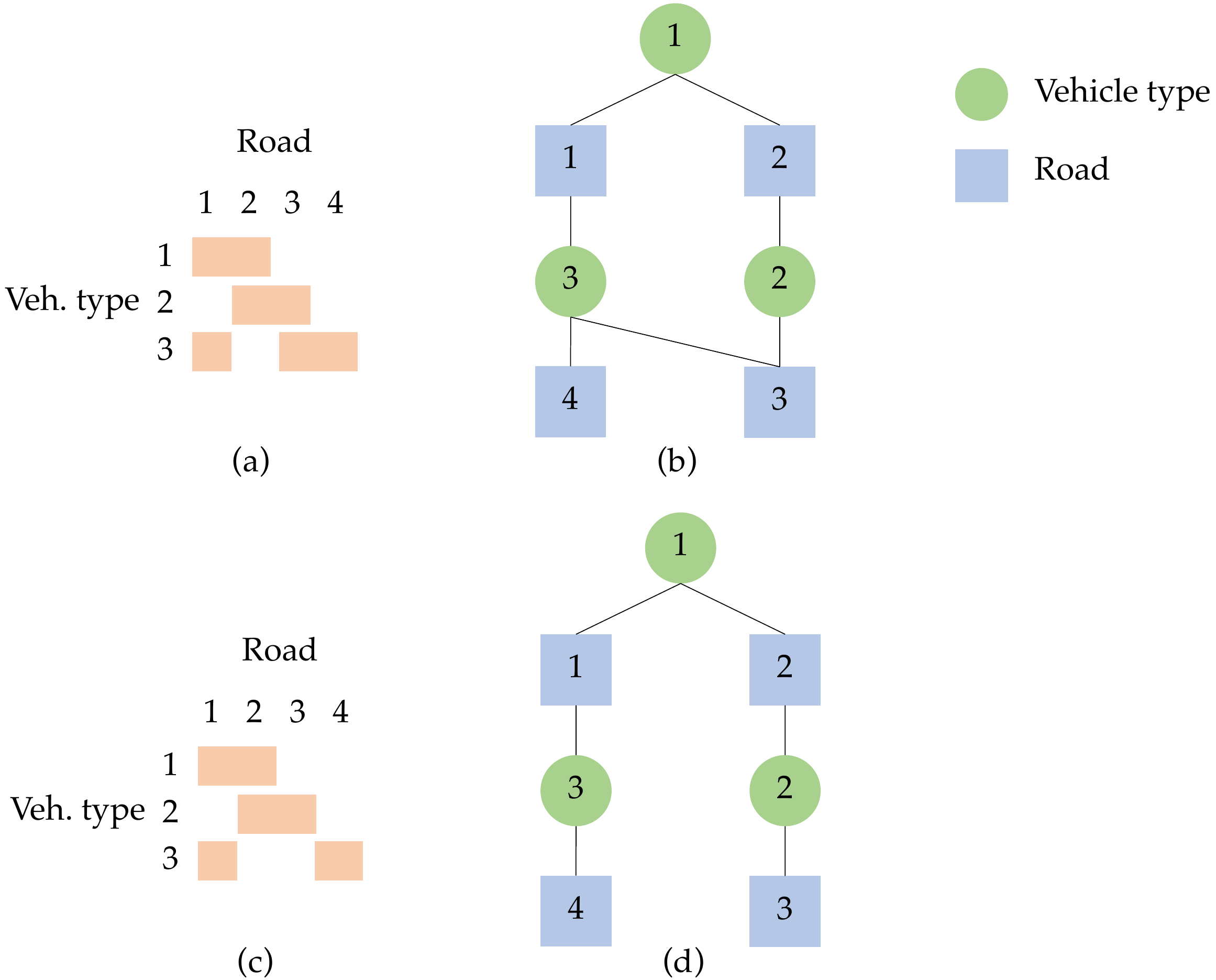}
	\caption{Two example routings for a network with four roads and three vehicle types. In (a), vehicle type 1 has positive flow on roads 1 and 2, type 2 has positive flow on roads 2 and 3, and type 3 has positive flow on road 1, 3, and 4. (b) shows the corresponding bipartite graph. (c) shows a similar routing but type 3 has zero flow on road 3, and (d) shows its corresponding bipartite graph.}
	\label{fig:cyclic_acyclic}
\end{figure}

%% file: results.tex
In this section we establish tolls which ensure that the social cost is minimized in any resulting equilibrium. We do this in two theorems: the first establishes properties about some routing in the set of routings which minimize the social cost, and the second provides optimal tolls (which are constructed based on the routing which is proved to exist in the first theorem) and proves their optimality. We begin with the first result.

\begin{theorem}\label{thm:routing}
	There exists a routing in the set of routings minimizing social cost, $f \in \mathcal{F}^*$, such that $G(f)$ is acyclic, where $G=(U,V,E)$ is constructed as in \eqref{eq:graph}, \emph{i.e.} where nodes are the roads and vehicle types, and edges exist between road $i$ and vehicle type $j$ when $f^j_i>0$.
\end{theorem}
\begin{proof}
	We prove this theorem constructively. We show that if there exists a routing in the set $\mathcal{F}^*$ that has a cyclic bipartite graph, we can break each cycle without altering the cost. We do this by showing that for a cyclic routing $f'$, there exists a feasible direction $d$ and that moving in the direction $d$ will not alter the cost and will eventually break the cycle.

The second-order partial derivatives of \eqref{eq:objective} are as follows
\begin{equation*}
\frac{\partial^2 J }{\partial f^j_i \partial f^{j'}_{i'} }(f) = \begin{cases} 0 &\mbox{if } i \neq i' \\
a^j_i + a^{j'}_i &\mbox{otherwise} \; .
\end{cases}
\end{equation*}

Since we define $f = \begin{bmatrix}f^1_1, f^2_1, \ldots, f^m_1, f^1_2, f^2_2, \ldots, f^m_n \end{bmatrix}^T$, the Hessian matrix is therefore block-diagonal in the following form:
\begin{equation*}
H = \begin{bmatrix}
H_1 & 0 & \ldots & 0 \\
0 & H_2 & \ldots & 0 \\
\vdots & \vdots & \ddots & \vdots \\
0 & 0 & \ldots & H_n
\end{bmatrix} \; ,
\end{equation*}
where the block corresponding to road $i$ is
\begin{equation*}
H_i = \begin{bmatrix}
2a^1_i & a^1_i + a^2_i & \ldots & a^1_i + a^m_i \\
a^1_i + a^2_i & 2a^2_i & \ldots & a^2_i + a^m_i \\
\vdots & \vdots & \ddots & \vdots \\
a^1_i + a^m_i & a^2_i + a^m_i & \ldots & 2a^m_i
\end{bmatrix} \; .
\end{equation*}

Consider a routing $f'$ which minimizes social cost and has a corresponding \emph{cyclic} graph. We will alter this routing and break the cycle while maintaining the same social cost. 

Since the graph is bipartite, any cycle will have the same number of nodes of each type. For some cycle in $f'$, let us use $r \in \{2,3, \ldots, \min(m,n)\}$ to denote the number of roads (and therefore vehicle types as well) in a simple cycle. The vehicles and roads are indexed arbitrarily, so let us consider the cycle to be comprised of the first $r$ roads (roads $\{1,2,\ldots,r\}$ and the first $r$ vehicle types (vehicle types $\{1,2,\ldots,r\})$. Let road $1$ be shared between vehicle types $1$ and $2$, road $2$ be shared between vehicle types $2$ and $3$, and so on, until road $r$ which is shared between types $r$ and $1$. The remaining roads and vehicle types are indexed arbitrarily.

Based on this feasible routing $f'$, we construct another routing $\tilde{f}$ which is also feasible. Consider flow $f' + \alpha d$, where $\alpha \in \mathbb{R}_{\ge 0}$ and $d$ is a direction vector as follows. As with $f$, let $d = \begin{bmatrix} d_1, d_2, \ldots, d_n \end{bmatrix}^T$, where $d_i$ corresponds to the flow change on road $i$, specifically  $d_i = \begin{bmatrix} d^1_i, d^2_i, \ldots, d^m_i \end{bmatrix}^T$. We choose $d_1 =  \begin{bmatrix} -1, 0, \ldots, 0, 1, 0, \dots 0 \end{bmatrix}^T$, where $d^r_1=1$. We choose $d_2 = \begin{bmatrix} 1, -1, 0, \ldots, 0 \end{bmatrix}^T$, and, for $i = \{3, \ldots, r \}$, $d_i$ is equal to $d_{i-1}$ circularly shifted downward. For $i \notin {[}r{]}$, $d^j_i=0$ for all  $j  \in \vtypeset$. 

The direction defined above corresponds to shifting some flow of type $1$ from road $1$ to road $2$, some flow of type $2$ from road $2$ to road $3$, and so on, ending in some flow of type $r$ shifting from road $r$ to road $1$. With this defined direction, the flow vector $f' + \alpha d$ is feasible for some range $\alpha \in [0, \overline{\alpha}]$, which we show as follows, using Definition~\ref{def:feas} (feasible flow).

When starting from a feasible flow, moving in the direction $d$ satisfies conservation of flow, as $\sum_{i \in \roadset}d^j_i=0$ for all $j \in \vtypeset$. Moreover, by the definition of flow $f'$, $d^j_i < 0$ only when $f'^j_i >0$, meaning that in the direction $d$, flow of a certain vehicle type on a road is decreased only if there exists positive flow of that vehicle type already. As such, we find the maximum feasible range to be $\overline{\alpha} = \min_{k\in {[}r{]}}f^k_k$, since at this point the nonnegativity constraint becomes active. The reason it is the minimum of $f^k_k$, with $k\in {[}r{]}$, is because in direction $d$, some of vehicle type $k$ is shifted off road $k$. Since at $\alpha = \overline{\alpha}$ the nonnegativity constraint becomes active, at the flow $\tilde{f} =  f' + \overline{\alpha} d$, the cycle on roads ${[}r{]}$ has been broken. Note that no new cycles have been induced, as $d^j_i$ can only be greater than zero when $f'^j_i$ is already greater than zero, meaning no new edges are added to the graph by moving in direction $d$.

We now investigate the social cost at $\tilde{f}$. Since the objective function is quadratic (and therefore analytic), by a trivial application of Taylor's inequality we can express $J(\tilde{f})$ as a second-order Taylor expansion around $f'$ as follows.
$$ J(\tilde{f}) = J(f') + \langle \nabla J(f'),\overline{\alpha}d \rangle + \frac{1}{2}\overline{\alpha}d^T H(f') \overline{\alpha}d \; ,$$
where $H(f')$ is the Hessian of $J$ evaluated at $f'$. First-order optimality conditions imply that $\langle \nabla J(f'),\overline{\alpha}d \rangle = 0$, since $d$ only has nonzero elements where inequalities for feasible flow are not tight; thus for $f'$ to be optimal, the derivative of $J(f')$ in the direction of $d$ must be zero. We now investigate the latter term, ignoring the scalar $\frac{\overline{\alpha}^2}{2}$. Since $H$ is block-diagonal,
\begin{align*}
	d^T H d = \sum_{i \in \roadset} d^T_i H_i d_i \;.
\end{align*}

Let us inspect $d^T_i H_i d_i$. For any specific road $i'$, $d_{i'}$ has one entry that is $1$ and one that is $-1$. Let us assign $p$ and $p'$ such that $d^{p}_{i'} = 1$ and $d^{p'}_{i'} = -1$. Then,
\begin{align*}
	d_{i'}^T H_{i'} d_{i'} &= \sum_{j' \in \vtypeset} d^{j'}_{i'} \sum_{j \in \vtypeset}(a^{j'}_{i'} + a^j_{i'})d^j_{i'} \\
	&= \sum_{j' \in \vtypeset}d^{j'}_{i'} [(a^{j'}_{i'} + a^p_{i'})- (a^{j'}_{i'} + a^{p'}_{i'})] \\
	&= \sum_{j' \in \vtypeset} d^{j'}_{i'}(a^p_{i'} - a^{p'}_{i'}) = (a^p_{i'} - a^{p'}_{i'}) - (a^p_{i'} - a^{p'}_{i'}) \\
	&= 0 \; .
\end{align*}
Thus, $d^T H d = 0$ as well, so
$$J(\tilde{f}) = J(f') \; .$$

We have thus constructed a flow $\tilde{f}$ which has the same social cost as a socially optimal flow $f'$, which has broken a cycle in the graph representing $f'$ without introducing a new cycle. Since the number of roads and vehicle types is finite, this process can be repeated until we arrive at a flow which has no cycles and has a social cost which optimizes \eqref{eq:objective}. This proves the theorem statement.
\end{proof}

\begin{corollary}
	No two vehicle types share more than one road with positive flow of both vehicle types.
\end{corollary}

We next provide a theorem describing tolls, based on the routing proven to exist above, which will lead the vehicles to follow the optimal routing.

\begin{theorem}\label{thm:toll}
	Consider a routing $\prevflow$ with an associated acyclic graph, which is proven to exist in Theorem~\ref{thm:routing}. Then levy the following tolls $\tau(\prevflow)$:
	\begin{equation}\label{eq:tolls}
		\tau^j_i(\prevflow) = \begin{cases}
		\mu- \ell_i(\prevflow) & \text{if} \; i \in \sharedroads^\prevflow_j \\	
		P & \text{otherwise} .
	\end{cases}
	\end{equation}
	Under Assumption~\ref{asmp:increasing_cost}, and for some constant $\mu$ and sufficiently large $P$, the only equilibrium that exists is the flow $\prevflow$.
\end{theorem}

\begin{proof}
We use the following properties related to the tolls described above.
\begin{property}\label{prop:toll_experience}
	As a result of the tolls in Theorem~\ref{thm:toll}, if two vehicle types have positive flow on a road in routing $\prevflow$, then the resulting tolls $\tau(\prevflow)$ will be such that the two vehicle types experience identical cost on that road. Mathematically,
	$$c^{j}_i(f) = c^{j'}_i(f) \; \forall f \in  \mathcal{F}_{\flowdemand} \land i \in \roadset \land j, j' \in \typesonroad^{\prevflow}_i \; .   $$
\end{property}

\begin{property}\label{prop:NE_eq_cost}
	As a result of Definition~\ref{def:NE}, if two roads have positive flow of a specific vehicle type at equilibrium, then the roads have equal cost for that vehicle type. Formally, if $f$ is an equilibrium, then
	$$
	c^j_i(f) = c^j_{i'}(f) \; \forall j \in \vtypeset \land i,i' \in \sharedroads^f_j \; .
	$$
\end{property}

\begin{property}\label{prop:toll_structure}
	For sufficiently large $P$, users in equilibrium will not use a road with toll $P$. Formally, for equilibrium flow $f$ experiencing tolls $\tau(\prevflow)$,
	$$
	f^j_i = 0 \; \forall j \in \vtypeset \land i \in \roadset \setminus \sharedroads^{\prevflow}_j \; .
	$$
\end{property}

We now prove the theorem by contradiction. Assume there exists some feasible flow $\newflow \neq \prevflow$ which is at equilibrium under tolls $\tau(\prevflow)$. Since $\newflow \neq \prevflow$, 
\begin{equation}\label{eq:branch1}
	\exists i \in \roadset \land j \in \vtypeset \; \text{s.t.} \; \newflow^j_i > \prevflow^j_i \; .
\end{equation}
From Property~\ref{prop:toll_structure}, $i \in \sharedroads^\prevflow_j$. Then, resulting from \eqref{eq:branch1}, either
\begin{enumerate}[(i)]
	\item $\exists j' \in \typesonroad^\prevflow_i \; \text{s.t.} \; \newflow^{j'}_i < \prevflow^{j'}_i $ , or
	\item $c^j_i(\newflow) > c^j_i(\prevflow)$ (from Assumption~\ref{asmp:increasing_cost}).
\end{enumerate}
If (i), then it must be the case that
\begin{enumerate}[(i)]
	\setcounter{enumi}{2}
	\item $\exists i' \in \sharedroads^\prevflow_{j'} \; \text{s.t. } \newflow^{j'}_{i'} > \prevflow^{j'}_{i'} $ , 
\end{enumerate}
due to Definition~\ref{def:feas} (flow conservation), which implies, from Definition~\ref{def:NE},
\begin{equation}
	c^{j'}_i(\newflow) \ge c^{j'}_{i'}(\newflow) \; .
\end{equation}
Then again, as a result of (iii) either (i) or (ii) must be the case, where $i$ is replaced by $i'$ and $j$ is replaced by $j'$. This process continues until we reach (ii). This terminal point must exist since the bipartite graph of roads and vehicle types is acyclic. Say the termination point is on vehicle type $j^{(k)}$ on road $i^{(p)}$. Then,
\begin{equation}\label{eq:1}
	c^{j^{(k)}}_{i^{(p)}}(\newflow) > c^{j^{(k)}}_{i^{(p)}}(\prevflow) = c^j_i(\prevflow) \; ,
\end{equation}
where the equality results from Properties~\ref{prop:toll_experience} and \ref{prop:NE_eq_cost}. Further, due to Definition~\ref{def:NE} and Property~\ref{prop:toll_experience},
\begin{equation}\label{eq:2}
	c^{j^{(k)}}_{i^{(p)}}(\newflow) \le c^{j^{(k)}}_{i^{(p-1)}}(\newflow) = c^{j^{(k-1)}}_{i^{(p-1)}}(\newflow) \le \ldots \le c^{j'}_i(\newflow) = c^j_i(\newflow) \; .
\end{equation}

We follow a similar logic down another branch. Reusing the indices $j'$ and $i'$ to a new use, we consider the other result of \eqref{eq:branch1}. By Definition~\ref{def:feas}, 
\begin{equation}\label{eq:branch2}
	\exists i' \in \sharedroads^\prevflow_j \; \text{s.t.} \; \newflow^j_{i'} < \prevflow^j_{i'} \; .
\end{equation}

Next, consider the results of \eqref{eq:branch2}. Similarly to above, either
\begin{enumerate}[(i)]
\setcounter{enumi}{3}
	\item $\exists j' \in \typesonroad^\prevflow_{i'} \; \text{s.t.} \; \newflow^{j'}_{i'} > \prevflow^{j'}_{i'} $ (where $j' \in \typesonroad^\prevflow_{i'}$ due to Property~\ref{prop:toll_structure}), or
	\item $c^{j'}_{i'}(\newflow) < c^{j'}_{i'}(\prevflow)$ (from Assumption~\ref{asmp:increasing_cost}).
\end{enumerate}
If (iv), then it must be the case that
\begin{enumerate}[(i)]
\setcounter{enumi}{5}
	\item $\exists i'' \in \sharedroads^\prevflow_{j'} \; \text{s.t. } \newflow^{j'}_{i''} < \prevflow^{j'}_{i''} $ ,
\end{enumerate}
due to Definition~\ref{def:feas} (flow conservation), which implies, from Assumption~\ref{asmp:increasing_cost},
\begin{equation}
	c^{j'}_{i''}(\newflow) \ge c^{j'}_{i'}(\newflow) \; .
\end{equation} 
Similarly to above, as a result of (vi), either (iv) or (v) must be the case, where $i'$ is replaced by $i''$ and $j$ is replaced by $j'$. This process continues until we reach (vi). Say the termination point is on vehicle type $j^{(r)}$ on road $i^{(s)}$. Then,
\begin{equation}\label{eq:3}
	c^{j^{(r)}}_{i^{(s)}}(\newflow) < c^{j^{(r)}}_{i^{(s)}}(\prevflow) = c^j_i(\prevflow) \; ,
\end{equation}
where the equality results from Properties~\ref{prop:toll_experience} and \ref{prop:NE_eq_cost}. Further, due to Definition~\ref{def:NE} and Property~\ref{prop:toll_experience},
\begin{equation}\label{eq:4}
	c^{j^{(r)}}_{i^{(s)}}(\newflow) \ge c^{j^{(r)}}_{i^{(s-1)}}(\newflow) = c^{j^{(r-1)}}_{i^{(s-1)}}(\newflow) \ge \ldots \ge c^{j'}_i(\newflow) = c^j_i(\newflow) \; . 
\end{equation}

We then string together \eqref{eq:1} and \eqref{eq:3} to find
\begin{equation*}
	c^{j^{(k)}}_{i^{(p)}}(\newflow) > c^{j^{(r)}}_{i^{(s)}}(\newflow) \; .
\end{equation*}
We also put together \eqref{eq:2} and \eqref{eq:4} to find
\begin{equation*}
	c^{j^{(k)}}_{i^{(p)}}(\newflow) \le c^{j^{(r)}}_{i^{(s)}}(\newflow) \; ,
\end{equation*}
which yields a contradiction.
\end{proof}

We show an example of this proof construction. Consider the routing in Figure~\ref{fig:cyclic_acyclic}~(c-d), with the proof illustrated in Figure~\ref{fig:ex_network_proof2}. As before, $\prevflow$ denotes the flow which minimizes the social cost and satisfies the condition in Theorem~1. We walk through a specific alternate flow $\newflow \neq \prevflow$ and show that it cannot exist in equilibrium when tolls are applied as in Theorem~\ref{thm:toll}, where the tolls are based on flow $\prevflow$.

Let us consider the flow $\newflow$ to have $\newflow^1_1 > \prevflow^1_1$, meaning the flow of type $1$ on road $1$ is higher in this new flow vector. By conservation of flow, $\newflow^1_2 < \prevflow^1_2$. Then, either $\newflow^3_1 < \prevflow^3_1$, meaning the flow of type $3$ on road $1$ decreases in the new flow vector, or  $c^1_1(\newflow) > c^1_1 (\prevflow)$ (by Assumption~\ref{asmp:increasing_cost}). Consider the former case. Then by conservation of flow and the fact that $\newflow$ must be in equilibrium and the tolling structure prevents flow of type $3$ on any road outside of the set $\{1,4\} $, $\newflow^3_4 > \prevflow^3_4$. Then by Assumption~\ref{asmp:increasing_cost}, $c^3_4(\newflow) > c^3_4(\prevflow)$.

Following the other branch of the diagram, as a result of $\newflow^1_2 < \prevflow^1_2$, either $\newflow^2_2 > \prevflow^2_2$ or $c^1_2(\newflow) < c^1_2(\prevflow)$. If the former, again $\newflow^2_3 < \prevflow^2_3$, resulting in $c^2_3(\newflow) < c^2_3(\prevflow)$.

\begin{figure}
	\centering
	\includegraphics[width=0.7\linewidth]{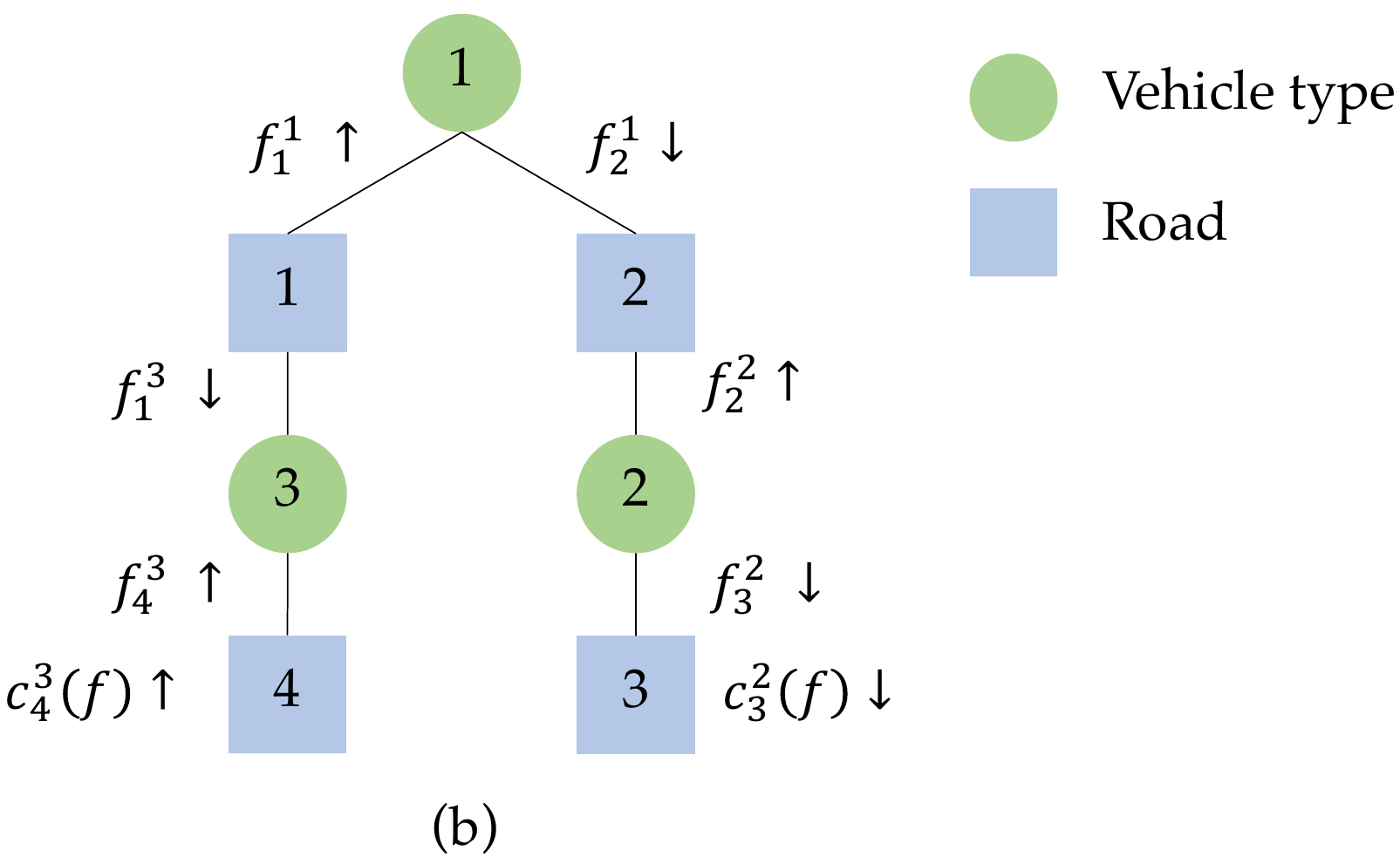}
	\caption{Flow changes for the proposed new flow at equilibrium.}
	\label{fig:ex_network_proof2}
\end{figure}

Due to the Properties~\ref{prop:toll_experience} and \ref{prop:NE_eq_cost},
$$
	c^3_4(\prevflow) = c^3_1(\prevflow) = c^1_1(\prevflow) = c^1_2(\prevflow) = c^2_2(\prevflow) = c^2_3(\prevflow) \; ,
$$
and combining it with the inequalities above, $c^3_4(\newflow) > c^2_3(\newflow)$. However, since $\newflow$ must be at equilibrium (yielding the inequalities), and from Property~\ref{prop:toll_experience} (yielding the equalities),
$$
	c^3_4(\newflow) \le c^3_1(\newflow) = c^1_1(\newflow) \le c^1_2(\newflow) = c^2_2(\newflow) \le c^3_2(\newflow) \; ,
$$
which yields a contradiction.

%% file: bounds.tex
In this section we show the necessity of the tolling scheme proposed in the previous section, both in that tolls must differentiate between vehicle types, and that well-known marginal cost tolling \cite{beckmann1956studies} fails in our setting.

\noindent \textbf{Undifferentiated tolls.} We show via example example that unless different vehicle types can experience differentiated tolls, the equilibrium social cost can be unboundedly worse than the social optimum. A previous work has shown that when a network has multiple source-destination pairs, undifferentiated tolls may not induce a socially optimal flow \cite{mehr2019pricing}. As in \cite{lazar2019optimal}, we extend these results to a simple two-road network and show that under undifferentiated tolling, the equilibrium can have a social cost which is arbitrary worse than the social optimum.

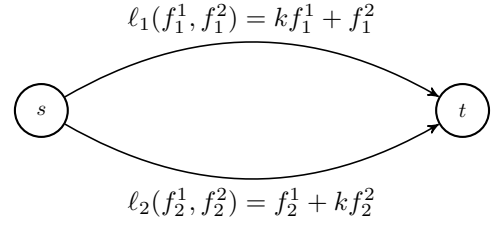
\begin{figure}
	\centering
	\begin{tikzpicture}[->, >=stealth', auto, semithick, node distance=7cm]
	\tikzstyle{every state}=[fill=white,draw=black,thick,text=black,scale=0.8]
	\node[state]    (0)               {$s$};
	\node[state]    (1)[right of=0]   {$t$};
	\path
	(0) edge[bend left]		node{$\ell_1(f^1_1, f^2_1) = k f^1_1 + f^2_1$}     (1)
	(0) edge[bend right]	node[below]{$\ell_2(f^1_2,f^2_2) = f^1_2 + k f^2_2$}     (1);
	\end{tikzpicture}
	\caption{Example of the futility of undifferentiated tolling in a simple network. Consider flow demands $\flowdemand^1=1$ and $\flowdemand^2=1$. The equilibrium under the best undifferentiated toll may be arbitrarily worse than the socially optimal routing.}
	\label{fig:undiffentiatedtolls}
\end{figure}

Consider the network in Fig.~\ref{fig:undiffentiatedtolls}, with flow demands $\flowdemand^1=1$ and $\flowdemand^2=1$, and let $k \ge 1$.  The socially optimal routing has social cost $2$, with $f^1_1=0$ and $f^2_1=1$. Without loss of generality, we can consider a toll on just one of the roads, since only the difference between the tolls on the two roads will affect the equilibrium. This example is symmetric, so without loss of generality let the top road be the road with a positive toll. In the resulting worst-case equilibrium, the top road has some flow of type $1$ and the bottom road has the remaining flow of type $1$ and all the flow of type $2$. To investigate how well the best toll can do, we derive the following.
\begin{align*}
	&\min_{f^1_1 \in {[}0,1{]}} f^1_1\ell_1(f^1_1,0) + (1-f^1_1+1)\ell_2(1-f^1_1,1) \\ 
	& = \min_{f^1_1 \in {[}0,1{]}} k (f^1_1)^2 + (1-f^1_1+1)(1-f^1_1+k) \\
	&= \frac{7k + 3}{4}-\frac{1}{k+1} < 2k \; .
\end{align*}

Though the toll decreases the social cost from that of the worst-case equilibrium, the social cost still increases linearly with $k$, while the socially optimal cost is constant with respect to $k$. This shows that the optimal undifferentiated tolling can result in arbitrarily worse social cost than the socially optimal routing even in this simple setting.

\noindent \textbf{Marginal Cost Tolling.} We next show by example that the classic marginal cost tolling, which is shown to be optimal in the case of a single vehicle type on a general network \cite{beckmann1956studies}, is not optimal in the multitype case in a parallel network with affine latency functions. Consider the network in Figure~\ref{fig:mctolls}, with two vehicle types, where the vehicle types have flow demand $ \flowdemand^1=2 $ and $\flowdemand^2=3$.

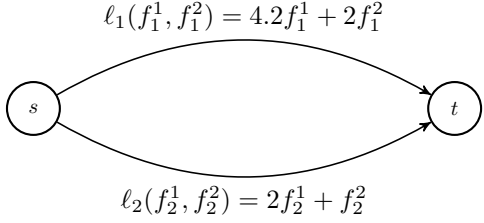
\begin{figure}
	\centering
	\begin{tikzpicture}[->, >=stealth', auto, semithick, node distance=7cm]
	\tikzstyle{every state}=[fill=white,draw=black,thick,text=black,scale=0.8]
	\node[state]    (0)               {$s$};
	\node[state]    (1)[right of=0]   {$t$};
	\path
	(0) edge[bend left]		node{$\ell_1(f^1_1, f^2_1) = 4.2 f^1_1 + 2 f^2_1$}     (1)
	(0) edge[bend right]	node[below]{$\ell_2(f^1_2,f^2_2) = 2 f^1_2 + f^2_2$}     (1);
	\end{tikzpicture}
	\caption{Example of the nonoptimality of marginal cost tolling in a simple network. Consider flow demands $\flowdemand^1=2$ and $\flowdemand^2=3$. Under marginal cost tolling, there exists an equilibrium with nonoptimal social cost.}
	\label{fig:mctolls}
\end{figure}

To find the optimal routing, we use Theorem~\ref{thm:routing} to solve four convex optimizations instead of one nonconvex one. The optimal routing is $f^1_1=0$ and $f^2_1=2$, for a social cost of $23$.

With marginal cost tolls, the toll for a vehicle type on a road is a function of the vehicle flows currently on that road -- in the affine case, $\tau^{j,\text{mc}}_i=a^j_i\sum_{j\in\vtypeset}f^j_i$. With these tolls, we find an equilibrium with $f^1_1=105/128$ and $f^2_1 = 95/128$; since both roads have positive flow of both vehicle types, we confirm that each vehicle type experiences identical cost on the two roads. In this equilibrium flow, the social cost is $23.6$, which is greater than the social optimum of $23$. This shows that marginal cost tolls are not optimal in this setting.

%% file: example.tex
We provide the following example to show the calculation of the tolling scheme described in the previous section, as well as the benefits of tolling. Consider a network of three parallel roads and three vehicle types, with flow demands $\overline{f}^1=3$, $\overline{f}^2=2$, and $\overline{f}^3=3$. Let the latency functions be as follows:
\begin{align*}
	\ell_1(f^1, f^2, f^3) &= 1 + 3f^1+f^2 + f^3 \\
	\ell_2(f^1, f^2, f^3) &= 2 + f^1 + 4f^2 + 2f^3 \\
	\ell_3(f^1, f^2, f^3) &= 1 + 2f^1 + f^2 + 3f^3 \; .
\end{align*}

In this example, a possible equilibrium has all flow of type $1$ on road $1$, flow of type $2$ on road $2$, and flow of type $3$ on road $3$, yielding a social cost of $80$. The optimal routing is shown in Table~\ref{table:example} and has flow type $1$ on roads $2$ and $3$, type $2$ on road $3$, and type $3$ on road $1$, for a social cost of $32.9$, approximately a $2.5$-fold improvement from the equilibrium described above. This routing satisfies the conditions in Theorem~\ref{thm:routing}\footnote{Some networks could also have socially optimal routings with cyclic bipartite graphs -- consider a network with $a^j_i=a^{j'}_i$ on all roads for two vehicle types $j$ and $j'$. In this case $j$ and $j'$ could share more than one road in an optimal routing. As guaranteed by Theorem~\ref{thm:routing}, there will also be routings in the set of minimizers in which the they share at most one road.}. The table also shows the proposed tolls based on this routing with $\mu=5$, which is chosen to be greater than the maximum latency so vehicles are only tolled, not subsidized.

\begin{table}
	\caption{Routings, Latency, and Tolls for the Example Network}
	\label{table:example}
	\centering
	\begin{tabular}{llllllllllll}
		\toprule
		& \multicolumn{4}{c}{Worst eq. routing} & \multicolumn{4}{c}{Opt. routing} & \multicolumn{3}{c}{Opt. Tolls}  \\
		& \multicolumn{4}{c}{social cost: $80$ } & \multicolumn{4}{c}{social cost: $32.9$} & &  \\
		\cmidrule(r){2-5}
		\cmidrule(r){6-9}
		\cmidrule(r){10-12}
		Rd  & $f^1$ \hspace{-10px} & $f^2$ \hspace{-10px} & $f^3$ \hspace{-10px} & $\ell$ & $f^1$ \hspace{-20px} & $f^2$ \hspace{-10px} & $f^3$ \hspace{-10px} & $\ell$ & $\tau^1$ & $\tau^2$ & $\tau^3$ \\
		\midrule
		1 & $3$ & $0$ & 0 & 10 & 0 & 0 & 3 & 4 & $\largenum$ & 1 & 1  \\
		2 & $0$ & $2$ & 0 & 10 & 2.83 & 0 & 0 & 4.83 & $\largenum$ & $\largenum$ & 0.17  \\
		3 & $0$ & $0$ & 3 & 10 & 0.17 & 2 & 0 & 4.33 & 0.67 & $\largenum$ & $\largenum$ \\
		\bottomrule
	\end{tabular}
\end{table}
\vspace{-0.1cm}

%% file: conclusion.tex
In this paper we considered tolling on parallel roads with multiple vehicle types and affine latency functions. We derived a key property of optimal routing in this setting and used this property to establish optimal tolls. Future works may generalize these results, both in terms of the network and the considered latency functions. Further extensions of this work will be critical in understanding traffic routing in the presence of many vehicle types with varying levels of autonomy.

%% file: root.bbl
\begin{thebibliography}{10}
\providecommand{\url}[1]{#1}
\csname url@samestyle\endcsname
\providecommand{\newblock}{\relax}
\providecommand{\bibinfo}[2]{#2}
\providecommand{\BIBentrySTDinterwordspacing}{\spaceskip=0pt\relax}
\providecommand{\BIBentryALTinterwordstretchfactor}{4}
\providecommand{\BIBentryALTinterwordspacing}{\spaceskip=\fontdimen2\font plus
\BIBentryALTinterwordstretchfactor\fontdimen3\font minus
  \fontdimen4\font\relax}
\providecommand{\BIBforeignlanguage}[2]{{%
\expandafter\ifx\csname l@#1\endcsname\relax
\typeout{** WARNING: IEEEtran.bst: No hyphenation pattern has been}%
\typeout{** loaded for the language `#1'. Using the pattern for}%
\typeout{** the default language instead.}%
\else
\language=\csname l@#1\endcsname
\fi
#2}}
\providecommand{\BIBdecl}{\relax}
\BIBdecl

\bibitem{lioris2017platoons}
J.~Lioris, R.~Pedarsani, F.~Y. Tascikaraoglu, and P.~Varaiya, ``Platoons of
  connected vehicles can double throughput in urban roads,''
  \emph{Transportation Research Part C: Emerging Technologies}, 2017.

\bibitem{mehr2018can}
N.~Mehr and R.~Horowitz, ``Can the presence of autonomous vehicles worsen the
  equilibrium state of traffic networks?'' in \emph{IEEE Conference on Decision
  and Control (CDC)}, 2018.

\bibitem{brown2019tragedy}
P.~N. Brown, ``A tragedy of autonomy. self-driving cars and urban congestion
  externalities,'' in \emph{IEEE Allerton Conference on Communication, Control,
  and Computing}, 2019.

\bibitem{roughgarden2002bad}
T.~Roughgarden and {\'E}.~Tardos, ``How bad is selfish routing?'' \emph{Journal
  of the ACM (JACM)}, 2002.

\bibitem{beckmann1956studies}
M.~Beckmann, C.~B. McGuire, and C.~B. Winsten, ``Studies in the economics of
  transportation,'' Tech. Rep., 1956.

\bibitem{mehr2019pricing}
N.~Mehr and R.~Horowitz, ``Pricing traffic networks with mixed vehicle
  autonomy,'' in \emph{American Control Conference (ACC)}, 2019.

\bibitem{lazar2019optimal}
D.~A. Lazar, S.~Coogan, and R.~Pedarsani, ``Optimal tolling for heterogeneous
  traffic networks with mixed autonomy,'' in \emph{IEEE Conference on Decision
  and Control (CDC)}, 2019.

\bibitem{gunter2020commercially}
G.~Gunter, D.~Gloudemans, R.~E. Stern, S.~McQuade, R.~Bhadani, M.~Bunting,
  M.~L. Delle~Monache, R.~Lysecky, B.~Seibold, J.~Sprinkle \emph{et~al.}, ``Are
  commercially implemented adaptive cruise control systems string stable?''
  \emph{IEEE Transactions on Intelligent Transportation Systems}, 2020.

\bibitem{dafermos1972multiclass_user}
S.~C. Dafermos, ``The traffic assignment problem for multiclass-user
  transportation networks,'' \emph{Transportation science}, 1972.

\bibitem{dafermos1969traffic_general}
S.~C. Dafermos and F.~T. Sparrow, ``The traffic assignment problem for a
  general network,'' \emph{Journal of Research of the National Bureau of
  Standards B}, 1969.

\bibitem{wardrop1900some}
J.~Wardrop, ``Some theoretical aspects of road traffic research,'' in
  \emph{Inst. Civil Engineers Proc. London, UK}, 1900.

\bibitem{depalma1998optimization}
A.~De~Palma and Y.~Nesterov, ``Optimization formulations and static equilibrium
  in congested transportation networks,'' Tech. Rep., 1998.

\bibitem{zhu2019routing}
Y.~Zhu and K.~Savla, ``On routing drivers through persuasion in the long run,''
  in \emph{IEEE Conference on Decision and Control (CDC)}, 2019.

\bibitem{wu2019information}
M.~Wu and S.~Amin, ``Information design for regulating traffic flows under
  uncertain network state,'' in \emph{IEEE Allerton Conference on
  Communication, Control, and Computing}, 2019.

\bibitem{correa2008geometric}
J.~R. Correa, A.~S. Schulz, and N.~E. Stier-Moses, ``A geometric approach to
  the price of anarchy in nonatomic congestion games,'' \emph{Games Econ.
  Behavior}, 2008.

\bibitem{perakis2007price}
G.~Perakis, ``The “price of anarchy” under nonlinear and asymmetric
  costs,'' \emph{Mathematics of Operations Research}, 2007.

\bibitem{lazar2020routing}
D.~A. Lazar, S.~Coogan, and R.~Pedarsani, ``Routing for traffic networks with
  mixed autonomy,'' \emph{To Appear, IEEE Transactions on Automatic Control
  (TAC)}, 2020.

\bibitem{dafermos1973toll}
S.~C. Dafermos, ``Toll patterns for multiclass-user transportation networks,''
  \emph{Transportation science}, 1973.

\bibitem{mehr2019will}
N.~Mehr and R.~Horowitz, ``How will the presence of autonomous vehicles affect
  the equilibrium state of traffic networks?'' \emph{arXiv preprint
  arXiv:1901.05168}, 2019.

\bibitem{lazar2017routing}
D.~A. Lazar, S.~Coogan, and R.~Pedarsani, ``Capacity modeling and routing for
  traffic networks with mixed autonomy,'' in \emph{IEEE Conference on Decision
  and Control (CDC)}, 2017.

\bibitem{askari2017effect}
A.~Askari, D.~A. Farias, A.~A. Kurzhanskiy, and P.~Varaiya, ``Effect of
  adaptive and cooperative adaptive cruise control on throughput of signalized
  arterials,'' in \emph{IEEE Intelligent Vehicles Symposium}, 2017.

\bibitem{bureau1964manual}
``Bureau of public roads traffic assignment manual,'' \emph{US Department of
  Commerce}, 1964.

\end{thebibliography}
